\documentclass[a4paper,11pt]{amsart}

\pdfoutput=1

\usepackage[T1]{fontenc}
\usepackage[utf8]{inputenc}

\usepackage{amssymb,amsfonts,amsmath,amsthm}
\usepackage{url}
\usepackage{color}
\usepackage{enumerate}
\usepackage{epsfig,graphicx,psfrag}
\usepackage{tikz}
\usepackage{asymptote}

\usepackage{mathabx}
\usepackage{MnSymbol}

\usepackage[top=1.2in,bottom=1.4in,left=1.4in,right=1.4in]{geometry}

\input{xy}
\xyoption{all}
\objectmargin={3mm}

\newcounter{notes}
\newcommand{\ignore}[1]{}

\definecolor{maxime}{rgb}{0,.6,0}
\definecolor{fred}{rgb}{0,0,0.8}

\newtheorem{theorem}{Theorem}
\newtheorem{proposition}[theorem]{Proposition}
\newtheorem{corollary}[theorem]{Corollary}
\newtheorem{lemma}[theorem]{Lemma}

\newtheorem{question}[theorem]{Question}

\theoremstyle{definition}
\newtheorem{definition}[theorem]{Definition}
\newtheorem{remark}[theorem]{Remark}

\newtheoremstyle{theoremwithref}{}{}{\itshape}{}{\bfseries}{.}{.5em}{#1 #2 #3}
\theoremstyle{theoremwithref}

\newcommand{\Z}{\mathbb{Z}}

\newcommand{\T}{\mathbb{T}}

\newcommand{\cA}{\mathcal{A}}

\newcommand{\cR}{\mathcal{R}}

\newcommand{\cW}{\mathcal{W}}

\title[Expansive partially hyperbolic diffeomorphisms]
{Expansive partially hyperbolic diffeomorphisms with one-dimensional center}

\author[M. Sambarino]{Martin Sambarino$^*$}
\thanks{$^*$Partially supported by CSIC-Universidad de la República-Uruguay  project \textit{Estructuras topológicas de sistemas parcialmente hiperbóicos y aplicaciones.}}
\address{CMAT, Facultad de Ciencias, Universidad de la Rep\'{u}blica, Uruguay}
\email{samba@cmat.edu.uy}

\author[J. Vieitez]{José Vieitez}
\address{DMEL, CenUR Literal Norte, Universidad de la República}
\email{jvieitez@litoralnorte.udelar.edu.uy}

\begin{document}

\numberwithin{theorem}{section}

\begin{abstract}
  We give sufficient conditions for an expansive partially hyperbolic diffeomorphism with one-dimensional center to be (topologically) Anosov.

\end{abstract}

\maketitle

\section{Introduction}\label{s.intro}

Partially hyperbolic diffeomorphisms appeared as a first kind of generalization of Anosov diffeomorphisms and it has been an active research area in the past decades. Moreover, during the last years there have been a lot of progress in classification results of partially hyperbolic diffeomorphisms in three manifolds (see \cite{P} and references therein).

Anosov diffeomorphisms have a fundamental property: expansiveness. Property that
is no longer true, in general, for partially hyperbolic diffeomorphisms.  Expansive systems on surfaces were classified (\cite{L},\cite{H1}) and great progress was made in three manifolds as well (\cite{V1}, \cite{V2}).

Our goal is to put together these two notions (partial hyperbolicity and expansiveness) and give some characterization. Before we state our main result let us recall the definitions of the above concepts.

A diffeomorphism $f:M\to M$ from a compact boundaryless manifold $M$ to itself is said to be partially hyperbolic if the tangent bundle  $TM=E^{ss}\oplus E^c\oplus E^{uu}$ splits into three invariant subbundles (stable, center and unstable) such that for some Riemannian metric there exists $\lambda, 0<\lambda< 1$ such that, for any $x\in M$ and unit vectors $v^{ss}, v^c, v^{uu}$ in $E^{ss}, E^c, E^{uu}$ respectively we have
$$\|Df_xv^{ss}\|<\lambda;\;\;\|Df_xv^{ss}\|<\lambda \|Df_xv^c\|<\lambda^2\|Df_xv^{uu}\|\;\mbox{ and }\|Df_xv^{uu}\|>\frac{1}{\lambda}.$$
We say that $f$ is Anosov if $E^c=\{0\}$ or if $E^c$ has a uniform -contracting or expanding- behavior (it could be the case where $E^c$ splits into two subbundles with uniform behavior as well). In the latter, the Anosov diffeomorphism can be regarded as a partially hyperbolic one as well.  It is well known that the existence of the bundles $E^{ss}, E^{uu}$ imply the existence of two invariant foliations $\cW^{ss}, \cW^{uu}$ called the stable and unstable ones. A partially hyperbolic diffeomorphism is said to be dynamically coherent if the bundles $E^{cs}:=E^{ss}\oplus E^c$ and $E^{cu}:=E^c\oplus E^{uu}$ give rise to invariant foliations $\cW^{cs}$ and $\cW^{cu}$ called  center-stable and center-unstable respectively. In this case $E^c$ is also tangent to an invariant foliation $\cW^c=\cW^{cs}\cap\cW^{cu}$ called the center one. We denote by $\cW^*_\varepsilon(x)$ the intrinsic ball of radius $\varepsilon$ centered at $x$ on the leave $\cW^*.$

A diffeomorphism $f:M\to M$ (or more generally a homeomorphism $f:M\to M$ from a compact metric space into itself) is said to be expansive if there exists $\alpha>0$ such that
$$\mbox{ if } dist(f^n(x),f^n(y))\le\alpha \mbox{ for all }n\in\Z \mbox{ for some } x,y\in M \mbox{ then } x=y.$$
In this case $\alpha$ is said to be an expansivity constant of $f.$

What can be said of an expansive partially hyperbolic diffeomorphism? In this generality it would be almost impossible to say something interesting, since taking any expansive diffeomorphisms on a compact manifold $g:N\to N$, one could find an Anosov diffeomorphism $f:\T^d\to\T^d$ such that $f\times g:\T^d\times N\to \T^d\times N$ is an expansive partially hyperbolic diffeomorphism.  A possible  approach would be to impose some restrictions, for instance on the dimension of the center bundle.

Is an expansive partially hyperbolic diffeomorphism with one-dimensional center an Anosov diffeomorphism? A simple example answers negatively this question: just take any Anosov diffeomorphism  which is partially hyperbolic with one-dimensional center (for instance a linear Anosov automorphisms on $\T^3$ with three different eigenvalues) and modify it so that in a fixed or periodic point it has one as an eigenvalue (but keeping that the fixed point is topologically hyperbolic and with the same stable  index, i.e., the dimension of the stable manifold): if it is done carefully it would be a partially hyperbolic diffeomorphism with one-dimensional center and conjugated to an Anosov diffeomorphism, but it is not an Anosov one (see \cite{M1} or \cite{PS}).


We say a diffeomorphism $f:M\to M$ is \textit{topologically Anosov} if there are two non trivial (topologically) transversal foliations $\cW^s$ and $\cW^u$ of complementary dimension such that for any $x\in M$ the leaves through $x$ are:
$$\cW^s(x):=\{y\in M: dist(f^n(x),f^n(y))\to_{n\to+\infty}0\},$$ $$\cW^u(x):=\{y\in M: dist(f^n(x),f^n(y))\to_{n\to-\infty}0\}.$$

{\begin{remark} \label{r.AH def}
Aoki and Hiraide, see \cite{AH}, define {\em Topological Anosov} to a homeomorphism which is expansive and has the pseudo orbit tracing property POTP (see Definition \ref{d.potp}).
This differs from the previous definition but in our context  both are equivalent (see Corollary \ref{c.equivalent}).
\end{remark}

\begin{question}\label{q,anosov}
If $f:M\to M$ is a topologically Anosov diffeomorphism,  is it true that it is conjugated to an Anosov one?
\end{question}
The answer is true for topological Anosov  defined on $\mathbb{T}^n$ (see \cite{H2}) or in a nilmanifold (see \cite{D}). The question in general could be hard, our question goes in the direction that if the stable and unstable sets form a true pair of complementary foliations may help to answer Question \ref{q,anosov} by finding an appropriate riemannian metric.

Our goal is to prove the following:

\begin{theorem}\label{t.main}
Let $f:M\to M$ be an expansive  partially hyperbolic diffeomorphism with one-dimensional center and dynamically coherent. Assume that one among the following conditions hold:
\begin{enumerate}
\item $\Omega(f)=M.$
\item The (topological) stable index of all the periodic points are the same.
\item $dim E^{uu}=1$ or $dim E^{ss}=1.$
\item $dim M\le 4.$
\item $f$  has the POTP on $M.$
\end{enumerate}
Then, $f$ is topologically Anosov.
\end{theorem}

 \vskip 5pt
Recall that a diffeomorphism $f:M\to M$ is quasi-Anosov if  $\|Df_x^n v\|\to\infty$ as $n\to +\infty$ or $n\to -\infty$ for any $x\in M$ and any $v\in T_xM-\{0\}.$ Mañé \cite{M2} showed that quasi-Anosov diffeomorphisms are expansive, satisfy Axiom A with $T_x\cW^s(x)\cap T_x\cW^u(x)=\{0\}$ for any $x\in M.$ We will get also the following:

\begin{corollary}\label{c.qa}
Let $f:M\to M$ be a dynamically coherent partially hyperbolic diffeomorphism with $1$-dimensional center and quasi-Anosov. Assume that one of the following hold:
\begin{enumerate}
\item $\Omega(f)=M.$
\item The stable index of all the periodic points are the same.
\item $dim E^{uu}=1$ or $dim E^{ss}=1.$
\item $dim M\le 4.$
\item $f$ has the POTP on $M.$
\end{enumerate}
Then, $f$ is  Anosov.
\end{corollary}

It follows trivially that the first two items imply that $f$ is Anosov from the characterization above ( \cite{M2}) (and there is no need of the partial hyperbolic structure).  We will show how the last three imply the corollary at the end of the following section.
\vskip 5pt
Regarding our proof (see Lemma \ref{l.atrep}) a question comes up.  Franks and Robinson  had  constructed a famous example of a quasi-Anosov diffeomorphism on a $3$-manifold which is not Anosov (see \cite{FR}).  It is expansive and has one hyperbolic attractor and one hyperbolic repeller.  However, this example cannot be partially hyperbolic by the structure of reebless foliation of codimension one in three manifolds (see \cite{HP}). See also Proposition \ref{p.newhouse} and the proof of Theorem \ref{t.main}.

\begin{question}
Is it possible to construct a Franks-Robinson type example of an expansive partially hyperbolic diffeomorphism with one-dimensional center having one hyperbolic attractor and one hyperbolic repeller on a manifold of dimension at least 5?
\end{question}

Our results implies that it cannot be constructed on a manifold with $dimM \le 4.$ There is some evidence that would imply that the answer to this question is NO. Indeed, we conjecture that \textit{any} expansive partially hyperbolic diffeomorphism with one-dimensional center must be topologically Anosov.
\vskip 5pt
\noindent\textbf{Acknowledgments:} We would like to thank Alexander Arbieto for proposing the problem to us. We would also like to thank fruitful conversations with S. Martinchich and R. Potrie and the anonymous referees who helped us to improve the paper and its presentation.

\section{Proof of Theorem \ref{t.main} and Corollary \ref{c.qa}}\label{s.proofmain}

We will prove Theorem \ref{t.main} based on the next proposition that we will prove later in Section \ref{s.proofprop}.  Through the whole section, $f$ will denote an expansive partially hyperbolic diffeomorphism with one-dimensional center and dynamically coherent, unless is explicitly stated. We denote by $\cW^{ss},  \cW^{uu}, \cW^c$ the stable, unstable and center foliations respectively. We denote by $I^c$ any arc contained in a center leaf, and by $|I^c|$ its length.

\begin{proposition}\label{p.options}
Let $f:M\to M$ be an expansive partially hyperbolic diffeomorphism with one-dimensional center and dynamically coherent. Let $I^c$ be any center arc, then one and only one of the following statements hold:
\begin{enumerate}
\item $\displaystyle{\lim_{n\to +\infty}|f^n(I^c)|=0}$ and $\displaystyle{\lim_{n\to -\infty}|f^n(I^c)|=\infty}.$
\item $\displaystyle{\lim_{n\to +\infty}|f^n(I^c)|=\infty}$ and $\displaystyle{\lim_{n\to -\infty}|f^n(I^c)|=0}.$
\item $\displaystyle{\lim_{n\to +\infty}|f^n(I^c)|=\infty}$ and $\displaystyle{\lim_{n\to -\infty}|f^n(I^c)|=\infty}.$
\end{enumerate}
Moreover, if $I_n^c$ is a sequence of center arcs converging to a (center) arc $J^c$  as $n\to\infty$ and $|f^m(I^c_n)|\to_m 0$ then $|f^m(J^c)|\to_m 0$ as $m\to +\infty$ or $m\to -\infty.$\footnote{This proposition says that $f$ is like a \textit{``topological quasi-Anosov''}.}

\end{proposition}
We now introduce the following two sets:

$$\cA^+=\{x\in M: \;\exists\, I^c \mbox{ with } x\in int(I^c) \mbox{ and } \sup_{n\ge 0}|f^n(I^c)| <\infty\}$$
$$\cA^-=\{x\in M: \;\exists\, I^c \mbox{ with } x\in int(I^c) \mbox{ and } \sup_{n\le 0}|f^n(I^c)| <\infty\}$$

The following lemma is straightforward:

\begin{lemma}\label{l.direct}
The following statements hold:
\begin{enumerate}
\item $\cA^+\cap \cA^-=\emptyset.$
\item $\cA^{\pm}$ are invariant sets
\item $\cA^+$ is saturated by stable leaves, i.e., if $x\in \cA^+$ then $\cW^{ss}(x)\subset \cA^+.$ The same for $\cA^-$ with unstable leaves.
\end{enumerate}
\end{lemma}

\begin{proof}
The first item is a direct consequence of expansiveness. Indeed, if $x\in \cA^+\cap \cA^-$ then there exists a center arc $I^c$ such that $|f^n(I^c)|$ is bounded for all $n\in\Z$ contradicting Proposition \ref{p.options}.

The second item is obvious. The third one is almost direct as well. Let $x\in \cA^+$ and $y\in \cW^{ss}(x).$ If we consider the stable holonomy inside $\cW^{cs}(x)$ between $\cW^c(x)$ and $\cW^c(y)$, the image of small arc $I^c$ by this holonomy is a center arc $J^c\subset \cW^c(y)$ with $y\in int(J^c).$ Iterating forwardly, it follows that $\lim_{n\to\infty}|f^n(I^c)|=\lim_{n\to\infty}|f^n(J^c)|.$
\end{proof}

\begin{lemma}\label{l.vacio}
If $\cA^+=\emptyset$ then $\cA^-=M$ and vice versa, if $\cA^-=\emptyset$ then $\cA^+=M.$
\end{lemma}

\begin{proof}
We will just prove one of the assertions since the other one is completely analogous.  Assume that $\cA^+=\emptyset$ and let $x\in M$ and a center arc $I^c$ with $x\in int(I^c).$ By Proposition \ref{p.options} we know that
$$\mbox{either (1) } \lim_{n\to -\infty}|f^n(I^c)|=0\;\;\mbox{ or (2) }\;\; \lim_{n\to -\infty}|f^n(I^c)|=\infty.$$
If $(1)$ holds for some $I^c$ then $x\in \cA^-.$ If not, then $(2)$ holds for every $I^c$ with $x\in int(I^c).$ Therefore, we get a sequence $I^c_n$ such that
\begin{itemize}
\item $|I ^c_n|\to_n 0.$
\item $\exists m_n\to\infty$ such that $|f^{-j}(I^c_n)|\le 1$ for $0\le j\le m_n$ and $|f^{-m_n}(I^c_n)|=1.$
\end{itemize}
Now, if $J^c$ is an accumulation arc of $f^{-m_n}(I^c_n)$ then $|f^n(J^c)|\le 1$ for $n\ge 0$ and so $\cA^+\neq\emptyset.$ A contradiction.
\end{proof}

\begin{corollary}\label{c.anosov}
If $\cA^+=\emptyset$ or
$\cA^-=\emptyset$ then $f$ is topologically Anosov.
\end{corollary}
\begin{proof}
By the previous lemma, if $\cA^+=\emptyset$, we get that
$$\cW^c(x)\subset \cW^u(x)=\{y\in M: dist(f^{-n}(x),f^{-n}(y))\to_{n\to\infty} 0\}.$$
In other words $\cW^{cu}(x)=\cW^u(x)$ implying that $f$ is (topologically) Anosov. The case $\cA^-=\emptyset$ is analogous with $\cW^{cs}(x)=\cW^s(x).$
\end{proof}

We now introduce two more sets:

$$\cR=\{x\in\cA^+: \cW^c(x)\subset\cA^+\}\;\;\;\mbox{ and }\;\;\;\;\;\cA=\{x\in\cA^-: \cW^c(x)\subset\cA^-\}.$$

\begin{lemma}\label{l.atrep}
The following statements hold:
\begin{enumerate}
\item If $\cA^+\neq\emptyset$ then the set $\cR$ is a nonempty compact invariant and  repelling set with $\cW^{s}(x)=\cW^{cs}(x)\subset \cR$ for any $x\in\cR.$
\item If $\cA^-\neq\emptyset$ then the set $\cA$ is a nonempty compact invariant and  attracting set with $\cW^{u}(x)=\cW^{cu}(x)\subset \cA$ for any $x\in\cA.$

\end{enumerate}

\end{lemma}
\begin{proof}
We will prove just the first item since the second one is completely similar. Let $x_n\in\cR$ be such that $x_n\to x.$ Let $y\in \cW^c(x).$ We may consider a center arc $J^c$ having $x$ and $y$ in its interior and a sequence $I^c_n\subset \cW^c(x_n)$ such that $I^c_n\to_n J^c.$ It follows by the last sentence of Proposition \ref{p.options} that $J^c\subset \cA^+.$ Since $y$ was arbitrary we get that $\cW^c(x)\subset \cA^+$ and so $\cR$ is compact.

Besides, if $x\in \cA^+$ then its $\alpha$-limit set $\alpha(x)\subset\cR.$ Indeed, if $x\in int(I^c)$ with $|f^n(I^c)|\to_{n\to+\infty}0$ and $f^{-m_k}(x)\to z$ then 
$f^{-m_k}(I^c)\to_k \cW^c(z)$ (on compact subsets) and, by Proposition \ref{p.options} again, we conclude that $\cW^c(z)\subset \cA^+.$ We conclude that $\cR$ is nonempty provided $\cA^+$ is nonempty.

 Finally, if $x\in\cR$ we know that $\cW^c(x)\subset\cA^+$ and hence $\cW^c(x)\subset \cW^s(x).$ Since $\cA^+$ is saturated by stable leaves, the same happens for $\cR$ and so $\cW^s(x)=\cW^{cs}(x)\subset\cR$  for any $x\in\cR.$ Therefore $\displaystyle{U=\bigcup_{z\in\cR}\cW^{uu}_\varepsilon(z)}$ is a neighborhood of $\cR$ satisfying $\overline{f^{-1}(U)}\subset U$ and $\displaystyle{\bigcap_{n\ge 0}f^{-n}(U)=\cR.}$

\end{proof}

\begin{definition}\label{d.potp}
Let $f:M\to M$ be a homeomorphism on a metric space $M$ and let $K$ be an invariant set.  We say that $f$ has the \textit{Pseudo Orbit Tracing Property (POTP)} on $K$ if given $\beta>0$ there exists $\alpha>0$ such that any $\alpha$-pseudo orbit $\{x_n\}_{n\in\Z}\subset K$ (i. e., $dist(f(x_n),x_{n+1})<\alpha$ for all $n\in
Z$) is $\beta$-shadowed by an orbit $y\in K$, i.e., $dist(f^n(y),x_n)<\beta$ for all $n\in\Z.$
\end{definition}

The following lemma shows that the classical \textit{shadowing lemma} holds in $\cA\cup\cR.$ This lemma and Lemma \ref{l.sstable} follow classical arguments of hyperbolic dynamics applied to ``topological hyperbolic sets'',  see  also \cite{R}. We include a proof for the sake of completeness.

\begin{lemma}\label{l.potp}
Let $f:M\to M$ be as in our hypothesis  and let $\cA$ and $\cR$ be as before. Then, $f$ has the POTP on $\cA\cup\cR.$
\end{lemma}

\begin{proof}
The proof follows the classical lines of the proof of the shadowing lemma for hyperbolic sets. We just give an outline of it. We know that $\cW^{cs}(x)=\cW^s(x)$ for $x\in \cR$ and $\cW^{cu}(x)=\cW^u(x)$ for $x\in \cA.$ Given $\varepsilon>0$ (say less than the expansivity constant) there exists $N>0$ such that
\begin{itemize}
\item If $x\in\cA$ then $f^{-N}(\cW^{cu}_\varepsilon(x))\subset \cW^{cu}_{\varepsilon/2}(f^{-N}(x)).$
\item If $x\in\cR$ then $f^{N}(\cW^{cs}_\varepsilon(x))\subset \cW^{cs}_{\varepsilon/2}(f^{N}(x)).$
\end{itemize}
The existence of such $N$ is an immediate consequence of expansiveness.
Indeed, if we assume that such an $N$ does not exist, then for any $n$ there exist sequences $x_n\in \cA$ and $y_n\in\cW^{cu}_\varepsilon(x_n)$ such that  $f^{-n}(y_n)\notin \cW^{cu}_{\varepsilon/2}(f^{-n}(x)).$ Taking convergent subsequences from $\{f^{-n}(x_n)\}$ and $\{f^{-n}(y_n)\}$ we find points $z,w$ such that $\varepsilon\geq dist(z,w)\geq \varepsilon/2$ and for every $k\in\Z$, $dist(f^k(z),f^k(w))\leq \varepsilon$ contradicting expansiveness. A similar argument holds for $x\in\cR$.

Taking this into account, one can prove the shadowing lemma in the same lines as in the hyperbolic case. We will do it for the set $\cA,$ for the set $\cR$ is completely similar.

Given $\beta_1>0$ choose $2\varepsilon <\beta_1$ and $\alpha_1>0$ such that if $x,y\in\cA$ with $dist(x,y)<\alpha_1$ then $\cW^{cu}_{\varepsilon/2}(z)\cap \cW^{ss}_\varepsilon(y)\neq\emptyset$ for any $z\in \cW^{ss}_{\lambda\varepsilon}(x).$ Let $\{x_n\}_{n\ge 0}$ be an $\alpha_1$-pseudo orbit in $\cA$ for $f^N.$ Define $z_0:=x_0$ and by induction
$$z_i:=\cW^{cu}_{\varepsilon/2}(f^N(z_{i-1}))\cap \cW^{ss}_\varepsilon(x_i).$$
By induction one proves that $$f^{-kN}(z_n)\in \cW^{cu}_{\varepsilon/2}(z_{n-k}).$$ This is trivial for $k=0$. Now, for $k\ge 0$ we have that $f^{-kN}(z_n)\in \cW^{cu}_{\varepsilon}(f^N(z_{n-k-1}))$ and so the statement remains valid for $k+1.$ Therefore, for $y_n=f^{-nN}(z_n)$ we have that $dist(f^{kN}(y_n),x_k)<2\varepsilon$ for $0\le k\le n.$ An accumulation point $y$ of $y_n$ will satisfy that $dist(f^{jN}(y),x_j)\le 2\varepsilon$ for $j\ge 0.$ From this, one concludes that any $\alpha_1$-pseudo orbit for $f^N$ in $\cA$ is $\beta_1$-shadowed by a true orbit of $f^N$ in $\cA$ and yields the  shadowing  for $f.$

\end{proof}

Observe that if $f$ is topologically Anosov then it is expansive and the previous proof applies and we have the following:
\begin{corollary}\label{c.anosovpotp}
Let $f:M\to M$ be topologically Anosov. Then, it has the POTP on $M.$
\end{corollary}

\begin{remark}\label{r.index} The sets $\cA$ and $\cR$ behave as \textit{hyperbolic sets}. Indeed, the set $\Omega(f)\cap\cA\cup\cR$ admits a \textit{spectral decomposition} as in the Axioma A case, with dense set of periodic orbits and local product structure, see \cite[Theorem 3.4.4]{AH}. Moreover, the attracting set $\cA$ must contain a topologically hyperbolic transitive attractor, and $\cR$ must contain a topologically hyperbolic transitive repeller.
 The periodic points in $\cR$ are topologically hyperbolic with stable index $dim (E^{ss}\oplus E^c)$ and the periodic points in $\cA$ have stable index $dim E^{ss}.$
\end{remark}

The important fact we need regarding the above is the following:

\begin{lemma}\label{l.sstable}
 Let $x\in M.$
  \begin{enumerate}
  \item If $\omega(x)\subset \cA$ then, there exists $z\in\cA$ such that $x\in\cW^{ss}(z).$
   \item If $\alpha(x)\subset \cR$ then, there exists $y\in\cR$ such that $x\in\cW^{uu}(y).$
\end{enumerate}
  \end{lemma}

\begin{proof}
The proof of this lemma as well as the previous Remark \ref{r.index} , relies on Lemma \ref{l.potp}, i.e., the shadowing lemma for the sets $\cA$ and $\cR$.
Now the classical argument in hyperbolic theory applies to our lemma. Let $\beta$ be small an less than half the expansivity constant and take $\alpha$ such that any $\alpha$-pseudo orbit in $\cA$ is $\beta$-shadowed by an orbit in $\cA.$ Let $x$ be such that $\omega(x)\subset \cA.$ Let $n_0$ be such that for $n\ge n_0$ one can find $x_n\in \cA$ such that $\{x_n\}_{n\ge n_0}$ is an $\alpha$-pseudo orbit in $\cA$ and $dist (f^n(x),x_n)<\beta.$ Then, there is $z\in A$ such that $dist(f^{n_0+j}(x), f^j(z))<2\beta.$ By expansivity and the local product structure of the foliations $\cW^{ss}, \cW^{cu}$ we conclude that $f^{n_0}(x)\in\cW^{ss}(z)$ and hence $x\in\cW^{ ss}(f^{-n_0}(z)).$ The second item is analogous.

\end{proof}

\begin{lemma}\label{l.omega}
 Let $x\in M.$ Then there exists $y$ and $z$ in $\cW^c(x)$ such that $\alpha(z)\cup\omega(y)\subset\cA\cup\cR.$ \end{lemma}

\begin{proof}
 We will just prove that for any $x$ there exists $y\in\cW^c(x)$ such that $\omega(y)\subset\cA\cup\cR.$ The existence of $z$ is proven along the same lines. Assume that
\begin{equation}\label{e.point}
\exists \,y\in W^c(x) \text{ s.t.  for every center arc } J^c_y, y\in J^c_y, \text{ it holds that }   |f^n(J^c_y)|\to_n \infty.
\end{equation}

It follows that if  $z\in\omega(y)$ then $\cW^c(z)\subset \cA^-$ and so $z\in\cA.$ Indeed, let $J^c_z$ be a closed center arc with $z$ as an endpoint and set $K=|J^c_z|.$ We can find a sequence of center arcs $I^c_n$ containing $y$ with $|I^c_n|\to 0$ and a sequence $m_n\to\infty$ such that $f^{m_n}(I^c_n)\to J^c_z.$ We claim that there exists $M$ such that $|f^j(I^c_n)|\le M$ for $0\le j\le m_n$  and any $n.$ If the claim isn't true, taking subsequences if necessary, there exists $k_n, 0\le k_n\le m_n$ such that $|f^{k_n}(I^c_n)|\to\infty.$\footnote{Compare with the proof of the last sentence of Proposition \ref{p.options}.} Consider a partition of $I^c_n$ into sub-arcs $I^c_n=I^c_{n,1}\cup....\cup I^c_{n,\ell_n}$ satisfying
\begin{itemize}
\item $|f^j(I^c_{n,t})|\le K$ for $0\le j\le m_n$ and $1\le t\le \ell_n.$
\item  For any $1\le t\le \ell_n-1$ there exists $s_t, 0\le s_t\le m_n$ such that $|f^{s_t}(I^c_{n,t})|=K.$
\end{itemize}
It is clear that $\ell_n\to\infty$. In particular there exists $t_n, 1\le t_n\le \ell_n-1$ such that $|f^{m_n}(I^c_{n,t_n})|\to 0.$ Hence $m_n-s_{t_n}\to\infty$ and $s_{t_n}\to\infty$ as well. Taking a convergent subsequence of the center arcs $f^{s_{t_n}}(I^c_{n,t_n})$ we find a center arc $I^c$ such that $|f^j(I^c)|\le K$ for any $j\in \Z$, contradicting Proposition \ref{p.options}. This proves our claim. We conclude then that $|f^{-j}(J^c_z)|\le M$ for any $j\ge 0$  that is, the length of $J^c_z$   backwardly iterated remains bounded. Now, since any center arc containing $z$ can be subdivided into two center arcs, both having $z$ as an endpoint, by the argument above  we get that $\cW(z)\subset \cA^-$ and $z\in \cA.$ In other words, $\omega(y)\subset \cA.$

On the other hand, assume that such a point $y$ as in \eqref{e.point}  does not exist. We claim 
in this case that for any (closed) center arc $I^c\subset\cW^c(x)$ it holds that $|f^n(I^c)|$ is bounded for $n\ge 0.$ If this is true, then we have that $\cW^c(x)\subset \cA^+$ implying that $x\in\cR$ and thus $\omega(y)\subset\cR$ for any $y\in\cW^c(x).$

We proceed to prove the claim. By contradiction assume that there is a closed center arc $I^c$ such that $|f^n(I^c)|\to\infty.$ Let $p_i, i=1,2$ be the endpoints of $I^c$. For each endpoint $p_i$ of $I^c$ consider, if exists, a maximal center arc $I_i\subset I^c$ containing $p_i$ such that $|f^n(I_i)|$ remains bounded. Let $J^c$ be the closure of $I^c\backslash (I_1\cup I_2).$ This is a nontrivial center arc, the length of its future iterates is unbounded, and the same is true for any sub-arc containing an endpoint of $J^c.$ Indeed, the same holds for any sub-arc of $J^c.$ Otherwise, if there exists a center arc $L^c\subset J^c$ such that $|f^n(L^c)|$ is bounded for $n\ge 0,$ we may consider $L^c$ as a maximal one with this property. It cannot contain any endpoint of $J^c.$ By Proposition \ref{p.options} we get that $|f^n(L^c)|\to 0.$  Observe that for any $x,y\in L^c$ it holds that $\omega(x)=\omega(y)$, in particular for the endpoints of $L^c.$ Let $z\in\omega(x), \,x\in L^c.$ By the same argument as in the first part (since $z$ is in the $\omega$-limit set of both boundary points of $L^c$) we conclude that $z\in \cA.$ Therefore, by Lemma \ref{l.sstable}, we get that the whole arc $L^c$ belongs to a single strong stable leaf, which is impossible. Thus, for any sub-arc of $J^c$ the length of its future iterates is unbounded. Then, any $y\in int(J^c)$ satisfies \eqref{e.point}, a contradiction. This proves our claim and our lemma.

\end{proof}

The lemma above implies that the center-unstable leaves of $\cA$ and $\cR$ cover all the manifold. The same holds for center stable leaves:

\begin{corollary}\label{c.cover}
 Let $f:M\to M$ be an expansive partially hyperbolic diffeomorphism with one-dimensional center and dynamically coherent. Let $\cA$ and $\cR$ be the attracting an repelling sets defined above. Then,
 $$\bigcup_{x\in\cA\cup\cR} \cW^{cs}(x)=M,\;\;\; \bigcup_{x\in\cA\cup\cR} \cW^{cu}(x)=M.$$
\end{corollary}


We now turn our attention to the case when $f$ has the POTP on the whole $M.$\footnote{We don't need to use the method used by Hiraide \cite{H2} about \textit{generalized foliations} since our situation is much simpler and we can give a simple proof.}

\begin{lemma}\label{l.potpM}
Assume that $f$ has the POPT on $M.$ Then $M=\cA\cup \cR.$
\end{lemma}

\begin{proof}
 We  start by noticing  that for $\varepsilon$ small we always have around each point $x\in M$ local product neighborhoods $U_\varepsilon(x)\cong D^{cs}_\varepsilon\times D^u_\varepsilon$ and $V_\varepsilon(x)\cong D^{s}_\varepsilon\times D^{cu}_\varepsilon$. This follows from the partially hyperbolic structure and the dynamical coherence. It is easy to see that if $dist(f^n(x), f^n(y))<\varepsilon$ for $n\ge 0$ then $y\in \cW^{cs}_\varepsilon(x)$ and if $dist(f^n(x), f^n(y))<\varepsilon$ for $n\le  0$ then $y\in \cW^{cu}_\varepsilon(x).$

Let $x\in M. $ We will prove first that either $x\in\cA^+$ or $x\in\cA^-.$
Assume that $x\notin \cA^+.$ We orient locally $\cW^c_\varepsilon(x)$ and identify it with $(-\varepsilon,\varepsilon),$  $x$ being $0.$ For $\delta<\varepsilon$ we denote $I^{c,+}_\delta(x)$ the arc $[0,\delta)$ in $\cW^c_\varepsilon(x)$ and by $I^{c,-}_\delta(x)$ the arc $(-\delta,0].$

Since $x\notin \cA^+$ then for any $\delta>0$ it holds
$$|f^n(I^{c,+}_\delta(x))|\to _{n\to +\infty}\infty \;  \mbox{ or } \; |f^n(I^{c,-}_\delta(x))|\to _{n\to +\infty}\infty.$$
Let us first assume $|f^n(I^{c,+}_\delta(x))|\to _{n\to +\infty}\infty$ and let $\beta>0$ be much smaller than $\varepsilon$ and consider the corresponding $\alpha>0$ from the POTP property.
Let $0<\delta<\alpha$ and let $y\in I^{c,+}_\delta(x).$ Consider the pseudo orbit $\{x_n\}_{n\in\Z}$ given by $x_n=f^n(x)$ for $n\le -1$ and $x_n=f^n(y)$ for $n\ge 0.$ By the POTP there is $z$ such that $dist(f^n(z),x_n)<\beta$ for all $n\in \Z.$ Note that $z\in \cW^{cs}_\varepsilon(y)$ and $z\in \cW^{cu}_\varepsilon (x)$ and so $z\in \cW^c_\varepsilon (x).$ Moreover, $z\in (0,\varepsilon)$ which implies that the arc $I^c=[x,z]\subset \cW^c_\varepsilon(x)$ satisfies $|f^{n}(I^c)|\to_{n\to-\infty} 0.$

If we also have that $|f^n(I^{c,-}_\delta(x))|\to _{n\to +\infty}\infty$ for any $\delta>0$, then, arguing as above we can conclude that $x\in\cA^-$ as we wish. Otherwise, for some $\delta>0$ we have that $|f^n(I^{c,-}_\delta(x))|\to_{n\to+\infty} 0.$ But this will led to a contradiction. Indeed, let $\beta$ be much smaller than $\varepsilon$, $\alpha>0$ corresponding to the POTP , $\delta<\alpha/2$ and $z,w$ such that $z\in I^{c,+}_\delta(x), w\in I^{c,-}_\delta(x)$ with
$$|f^n([x,z])|\to_{n\to-\infty} 0\;\;\text{ and }\;\;|f^n([w,x])|\to_{n\to+\infty} 0.$$
Let $\{x_n\}_{n\in\Z}$ be the $\alpha$ pseudo orbit by $x_n=f^{-n}(w)$ for $n\le -1$ and $x_n=f^n(z)$ for $n\ge 0.$ There must be a point $t$ that $\beta$ shadows $x_n.$ As before, $t$ must be in $\cW^{c}_\varepsilon(x).$ However,  $t\ge x$ cannot be, since in the  past will get closer to $f^{-n}(x)$ and far from $f^{-n}(w)$ (recall from Proposition \ref{p.options} that $|f^n([w,x])|\to_{n\to-\infty}\infty$). On the other hand $t\le x$ cannot be either, arguing with the future and $f^n(y).$ We have reached a contradiction.

This proves that the assumption that $x\notin \cA^+$  leads to $x\in \cA^-$
proving that for any $x\in M$ either $x\in \cA^+$ or $x\in\cA^-.$ By connectedness of $\cW^c$ we conclude that for each $x\in M$ either $\cW^c(x)\subset \cA^+$ or $\cW^c(x)\subset \cA^-.$ In other words, $x\in \cA\cup \cR.$
\end{proof}

\begin{corollary}\label{c.equivalent}
Let $f:M\to M$ be an expansive partially hyperbolic diffeomorphism with one-dimensional center and  dynamically coherent. Then
$$f \text{ is topologically Anosov} \iff f \text{ has the POTP on }M$$
\end{corollary}
\begin{proof}
The only if part is Corollary \ref{c.anosovpotp}. The if part is a consequence of Lemma \ref{l.potpM} above and the fact that $\cA$ and $\cR$ are closed disjoint sets. By connectedness, either $\cA=M$ or $\cR=M$ and Corollary \ref{c.anosov} gives the result.
\end{proof}

We need the following proposition to prove our main theorem when $dim E^{uu}=1$. It is based on a Hiraide's argument \cite{H} of a Newhouse's Theorem (\cite{N}) :

\begin{proposition}\label{p.newhouse}
Let $f:M\to M$ be an expansive  partially hyperbolic diffeomorphism $TM=E^{ss}\oplus E^c\oplus E^{uu}$  with one-dimensional center and dynamically coherent. Assume that $dim E^{uu}=1$ and there is a compact repelling set $\cR\neq M$ such that if $x\in \cR$ then $\cW^{cs}(x)\subset\cR$ and
$$\cW^{cs}(x)=\cW^s(x):=\{y\in M: dist(f^n(y),f^n(y))\to_{n\to+\infty}0\}.$$
Then, given any neighborhood $U$ of $\cR$ there exists a point $y\in U\backslash \cR $ such that $\cW^{cs}(y)\subset U.$
\end{proposition}

\begin{proof} We will follow the lines of \cite{dMM} on Hiraide's argument.
We may assume that $E^{uu}$ is orientable and oriented, and so we have an orientation on the unstable foliation $\cW^u$ We have that $dim \cW^{cs}\ge 2$ and for points in $x\in\cR$ we know that $\cW^{cs}(x)$ is (diffeomorphic to) an euclidean space.

Let $U$ be a neighborhood of $\cR.$  For each $x\in\cR$ we can consider a product neighborhood $U_x\cong D^{cs}(x)\times [-1,1]\subset U$ where $D^{cs}(x)$ is a compact $cs$-ball such that for $t\in [-1,1]$ we have that $D^{cs}(x)\times \{t\}$ is a ball in the center-stable foliation and for $y\in D^{cs}$ we have that $\{y\}\times [-1,1]$ is an arc in $\cW^u(y).$

Since $\cR\neq M$ we may assume that $D^{cs}(x)\times \{-1,1\}\cap\cR=\emptyset.$ We can take a finite covering of $\cR$ by $\{U_{x_i}, i=0,1,\cdots, k\}$ with $U_{x_i}\subset U$ for every $i$. The repelling set $\cR$ has empty interior (otherwise is the whole of $M$) and so $\cR\cap\cW^u(x)$ has empty interior in $\cW^u(x)$ for any $x\in\cR.$ Therefore, there exist $x\in\cR$ and an unstable arc $I^u_x=(x,x^u)\subset \cW^u(x)$ in the orientation of $\cW^u$ such that $I^u_x\cap\cR=\emptyset.$ We may assume without loss of generality that $x=x_0$ and that $I^u_x=\{x_0\}\times (0,\delta)$ in $U_{x_0}.$

For each $i=0,1,\cdots, k$ let $t_i=sup\{t\ge 0: (x_i,t)\in\cR\}$ and denote by $P_i$ the plaque in $U_{x_i}$ through $(x_i,t_i).$ Let us consider a ball $E:=D^{cs}(x_0,r)\subset \cW^{cs}(x_0)$ such that $(\cW^{cs}(x_0)\backslash E) \cap P_i=\emptyset $ for every $i=0,1,...,k.$ In other words, if $(x_i,t_i)\in \cW^{cs}(x_0)$ then $P_i\subset E.$
Let $V=\overline{E}\times (-\varepsilon,\varepsilon)$ be a product neighborhood where $\overline{E}=E\cup \partial E$ is the closed ball and we may assume that $\{x_0\}\times (0,\delta)\subset V\subset U.$

Now, for each $x\in \cW^{cs}(x_0)\backslash E$ there exists an arc $L_x=[x,x+\delta_x] \subset \cW^u(x)$ such that $L_x\cap\cR=\{x,x+\delta_x\}$, i.e., $L_x$ intersects $\cR$ just in its endpoints. Indeed, since $x\in \cW^{cs}(x_0)$ it cannot be approximated by ``above'' (in the orientation of $E^{uu}$) by points in $\cR$, and since $x\notin E$ then it cannot belong to some $P_i$  and hence it must have a point of $\cR$ above it. These observations imply the existence of $L_x.$ Observe that if $x\in U_{x_i}$ then $L_x\subset U_{x_i}.$ Clearly $L_x$ depends continuously on $x.$ We may assume also that $\varepsilon$ is small so that $x\times (0,\varepsilon)\subset L_x$ for any $x\in\partial E.$

Observe that $\displaystyle{\bigcup_{x\in\cW^{cs}(x_0)\backslash E}L_x}$ is a a foliated interval bundle with base $\cW^{cs}(x_0)\backslash E$ and we have a well-defined projection $\displaystyle{\pi: \bigcup_{x\in\cW^{cs}(x_0)\backslash E}L_x\to \cW^{cs}(x_0)\backslash E}$ by $\pi (L_x)=\pi ([x,x+\delta_x])=x.$

Fix $x_0'\in\partial E$ and $z\in \{x_0'\}\times (0,\varepsilon)\subset L_{x_0'}.$ Let $x\in\cW^{cs}(x_0)\backslash E$ and $\gamma:[0,1]\to\cW^{cs}(x_0)\backslash E$ be a path with $\gamma(0)=x_0'$ and $\gamma(1)=x.$ We can lift $\gamma$ to $\gamma_z;[0,1]\to \cW^{cs}(z)$ such that $\pi\circ\gamma_z=\gamma.$
\vskip 5pt
\textit{Claim 1: The projection $p_\gamma:L_{x_0'}\to L_x$ by $p_\gamma(z)=\gamma_z(1)$ is independent of the path $\gamma:[0,1]\to \cW^{cs}(x_0)\backslash E$ joining $x_0'$ and $x.$ }

Recall that $dim \cW^{cs}\ge 2.$ Let $\gamma$ and $\gamma'$ be two path joining $x_0'$ and $x$ in $\cW^{cs}(x_0)\backslash E$. Since $\cW^{cs}(x_0)$ is contractible, the path $\gamma'^{-1}*\gamma$ from $x_0'$ to itself is homotopic (within $\cW^{cs}(x_0)\backslash E$) to a path $\alpha$ contained in $\partial E$ from $x_0'$ to itself. Clearly the lift of $\alpha$ to $\alpha_z$ in $\cW^{cs}(z)$ is a closed path (recall that $V=\overline{E}\times (0,\varepsilon)$ is a foliated neighborhood). Therefore $\alpha_z(1)=z$ and hence $\gamma_z(1)=\gamma'_z(1).$
\vskip 5pt

\textit{Claim 2: Let $y=(x_0,t)\in \{x_0\}\times (0,\delta)\subset V$ and let $z=(x_0',t)\in V$. Then $\cW^{cs}(y)\subset U.$}

By Claim 1 we have a well defined map $F:\cW^{cs}(x_0)\backslash E\to \cW^{cs}(y)=\cW^{cs}(z)$ by $F(x)=\gamma_z(1).$ This map can be extended to the whole $\cW^{cs}(x_0)$ by just defining $F(x)=(x,t)$ if $x\in E.$ Clearly $F$ is injective and a local homeomorphism. Moreover  $F(\cW^{cs}(x_0))$ is open and closed in $\cW^{cs}(y)$  and so it is surjective. It is obvious that $F(x)\in U$ for all $x\in\cW^{cs}(x_0).$ This proves the claim.

Since $y$ was chosen outside $\cR$ the proof of the proposition is complete.

\end{proof}

\noindent We have all the ingredients to proceed to the proof of our main result.
\vskip 5pt
\noindent \textbf{Proof of Theorem \ref{t.main}:}  Whatever the assumptions on Theorem \ref{t.main} one choose, one gets that either $\cR=\emptyset$ or $\cA=\emptyset.$ Indeed:\\
\noindent\textbullet\;  If $\Omega(f)=M$ we have that $\cA$ and $\cR$ cannot coexist and by Corollary \ref{c.anosov} we conclude that $f$ is topologically Anosov.

\noindent\textbullet\; If all the periodic points have the same topological stable index then, by Remark \ref{r.index} we see that either $\cR=\emptyset$ or $\cA=\emptyset$ and so either $\cA^+$ or $\cA^-$ are empty, and by Corollary \ref{c.anosov} we conclude that $f$ is topologically Anosov as well.

\noindent\textbullet\; Assume that $dim E^{uu}=1.$ If either $\cA^+=\emptyset$ or $\cA^-$ are empty  we are done as before. Otherwise, $\cA$ and $\cR$ are nonempty and so $\cR$ is as in Proposition \ref{p.newhouse}. We can choose a neighborhood $U$ of $\cR$ such that $U\cap\cA=\emptyset.$ By Proposition \ref{p.newhouse} we get the existence of a center-stable leaf $\cW^{cs}(y)$ such that $\cW^{cs}(y)\cap \left(\cA\cup\cR\right)=\emptyset. $ This contradicts Corollary \ref{c.cover}. If $dim E^{ss}=1$ we get the same result arguing with $f^{-1}.$

\noindent\textbullet\; If $dim M\le 4$ then it must hold that $dim E^{uu}$ or $E^{ss}$ has dimension one. 
\noindent\textbullet\; The case of the POTP on $M$ was done in Corollary \ref{c.equivalent}.

This concludes the proof of Theorem \ref{t.main}, assuming Proposition \ref{p.options}.

\vskip 5 pt

\noindent\textbf{Proof of Corollary \ref{c.qa}:} As we said in the introduction, for a quasi-Anosov diffeomorphism the fact that $\Omega(f)=M$ or that the periodic points have the same stable index imply that $f$ is Anosov by the work of Mañé in \cite{M2}. Now assume that $dim E^{uu}=1.$ If either $\cA+=\emptyset$ or $\cA^-=\emptyset$ we have that all periodic points have the same index and therefore $f$ is Anosov. Otherwise, $\cA$ and $\cR$ are nonempty. and $\cR$ is a codimension one repeller in the hypothesis of Proposition \ref{p.newhouse}. Let $U$ be a neighborhood of $\cR$ disjoint form $\cA.$ Proposition \ref{p.newhouse} give  a contradiction with Corollary \ref{c.cover} as before. The same argument works if $dim E^{ss}=1$ arguing with $f^{-1}.$ And if $dim M\le 4$ then either $dim E^{uu}=1$ or $dim E^{ss}=1$ and then $f$ is Anosov. Finally, if it has the POTP property then it is Anosov by the argument above, or just by recalling classical result that an Axiom A diffeomorphism with the POTP is structurally stable and so it has the (strong) transversality condition. Since it is expansive, it must be Anosov.

\section{Proof of Proposition \ref{p.options}}\label{s.proofprop}

Let $\alpha$ be an expansivity constant of our expansive partially hyperbolic diffeomorphism $f:M\to M.$ The proof of Proposition \ref{p.options} is done through several lemmas that use classical/folklore arguments on expansiveness.

\begin{lemma}\label{l.inf}
Let $I^c$ be a center arc such that $\liminf_{n\to +\infty}|f^n(I^c)|=0.$ Then $$\lim_{n\to +\infty}|f^n(I^c)|=0.$$
\end{lemma}
\begin{proof}
Otherwise, there exist $\eta>0$ and sequences $n_j, m_j\to \infty$ such that
\begin{itemize}
\item $|f^{n_j}(I^c)|<\frac{1}{j}.$
\item $|f^{m_j}(I^c)|>\eta.$
\item $n_j<m_j<n_{j+1}.$
\end{itemize}

It follows that $m_j-n_j\to\infty$ and $n_{j+1}-m_j\to\infty.$ Then, for each $j$, we can choose a maximal center arc $I^c_j\subset f^{n_j}(I^c)$ satisfying
$$|f^k(I^c_j)|\le \min\{\eta,\alpha\} \mbox{ for }0\le k\le n_{j+1}-n_j.$$ It follows that there exists $k_j$ such that $|f^{k_j}(I^c_j)|=\min\{\eta,\alpha\}$. Besides, it holds that $k_j\to\infty$ and $n_{j+1}-n_j-k_j\to \infty.$
Let $J^c$ be an accumulation arc of $f^{k_j}(I^c_j).$ Then $|f^n(J^c)|\le\alpha$ for every $n\in\Z$ contradicting the expansiveness of $f.$
\end{proof}

\begin{lemma}\label{l.bounded}
There exists $\delta>0$ such that if  $I^c$ is  a center arc with $$\limsup_{n\to+\infty}|f^n(I^c)|<\infty$$ then one and only one of the following is true:
\begin{enumerate}
\item $\lim_{n\to +\infty}|f^n(I^c)|=0.$
\item there exists $n_0$ such that $|f^n(I^c)|\ge \delta$ for $n\ge n_0.$
\end{enumerate}
\end{lemma}

\begin{proof}
We will argue by  contradiction and assume that such a $\delta$ does not exist. Then, using Lemma \ref{l.inf} there exist $\delta_n\to 0$ and center arcs $I^c_n$ such that:
\begin{itemize}
\item $\limsup_{n\to\infty}|f^k(I^c_n)|<\infty.$
\item $\liminf_{n\to\infty}|f^k(I^c_n)|=\delta_n.$
\end{itemize}
We may assume without loss of generality that $|f^k(I^c_n)|\ge \frac{\delta_n}{2}$ for every $k\ge 0$ and that $2\delta_n<\alpha$ for every $n.$ Notice that there exists a sequence $k_m(n)\to +\infty$ such that
$$|f^{k_m(n)}(I^c_n)|\le 2\delta_n<\alpha\;\;\; \forall m \mbox{ and } n.$$
Now, if $|f^k(I^c_n)|\le \alpha $ for all $k \ge j_0(n)$, for some $n$ we get a contradiction with expansiveness by taking an accumulation arc of $f^k(I^c_n)$ as $k\to \infty$ for this $n.$

Therefore, there exists $\tilde{k}_m(n), k_m(n)<\tilde{k}_m(n)<k_{m+1}(n)$ such that $$|f^{\tilde{k}_m(n)}(I^c_n)|>\alpha \mbox{ for all }m \mbox{ and }n.$$ It follows that $\tilde{k}_m(n)-k_m(n)\to _n\infty$ and $k_{m+1}(n)-\tilde{k}_m(n)\to_n\infty.$ Arguing as in Lemma \ref{l.inf} (and fixing $m$) we may find a center arc $I^c_{n_m}\subset f^{\tilde{k}_m(n)}(I^c_n)$ such that
$$|f^j(I^c_{n_m})|\le \alpha \mbox{ for } 0\le j\le k_{m+1}(n)-k_m(n) \mbox{ and }|f^{j_n}(I^c_{n_m})|=\alpha \mbox{ for some }j_n.$$
Taking an accumulation arc $J^c$ of $f^{j_n}(I^c_{n_m})$ we get a contradiction since $|f^j(J^c)|\le\alpha$ for all $j\in\Z.$
\end{proof}

\begin{lemma}\label{l.lim0}
Let $I^c$ be a center arc such that $\limsup_{n\to+\infty}|f^n(I^c)|<\infty.$ Then,
$$\lim_{n\to +\infty}|f^n(I^c)|=0.$$
\end{lemma}
\begin{proof}

 Let $I^c$ be a center arc as in the hypothesis and let $K>0$ be such that $|f^n(I^c)|\le K$ for $n\ge 0.$ Let $\delta>0$ from Lemma \ref{l.bounded} and consider $m_0\in\Z^+$ such that $m_0>\frac{K}{\delta}.$ Notice that if we could find $m_0$ distinct sub-arcs $I^c_1,....,I^c_{m_0}$ of $I^c$ with pairwise disjoint interior such that $|f^n(I^c_m)|\not\to_{n\to+\infty} 0$ for each $m=1,...,m_0,$  then for each $m$ there exists $k_m$ such that $|f^j(I^c_m)|\ge \delta$ for $j\ge k_m$, and  for $N>\max\{k_m: m=1,...,m_0\}$ we have that $$K<m_0\delta\le |f^N(I^c_1)|+...+|f^N(I^c_{m_0})|\le |f^N(I^c)|\le K,$$
implying a contradiction. Thus, the idea of the proof of the lemma is: if the thesis does not hold we could find such sub-arcs. Assume then that $|f^n(I^c)|\not\to_n 0.$

We may assume without loss of generality that $I^c$ is closed. Let us first show that there exists a point $x\in I^c$ such that for every sub-arc $\ell^c\subset I^c$ containing $x$ in its interior (with respect to $I^c$) we have $|f^n(\ell^c)|\not\to_{n\to+\infty} 0.$  This can be proved by bipartition of $I^c$. Dividing $I^c$ into two equal sub-arcs $I^c_1, I^c_2$ one of them has to have the property that  $|f^n(I^c_j)|\not\to_{n\to+\infty} 0$ ($j=1$ or $2$) otherwise $|f^n(I^c)|\to_{n\to+\infty} 0$. Dividing again a suitable sub-arc and repeating this procedure infinitely many times we found the required point $x$. The point $x$ could be a boundary point of $I^c$. In this case, any arc in $I^c$ containing $x$, it has it in its interior with respect to $I^c.$
\vskip 5pt
\noindent\textit{Claim: For any $r>0$ and any arc $\ell^c_r\subset I^c$ with $|\ell^c_r|=r$ and containing $x$ in its interior  there exists a closed arc $J\subset \ell^c_r$ with $x\notin J$ such that  $|f^n(J)|\not\to_n 0$.}

Otherwise, there exists an arc $\ell^c\subset I^c$ containing $x$ such that  \textit{any} closed arc $J\subset \ell^c$ with $x\notin J$ satisfies $|f^n(J)|\to_n 0.$ We will arrive to a contradiction with expansiveness. Indeed, choose a sequence $r_n\to 0$ and $\ell^c_{r_n}\subset \ell^c$ containing $x$ in its interior and $|\ell^c_{r_n}|=r_n.$ We know from Lemma \ref{l.bounded} that for some $m_n$ that $|f^{m_n}(\ell^c_{r_n})|\ge \delta.$ Consider a closed arc $J_{m_n}\subset f^{m_n}(\ell^c_{r_n})$ such that $f^{m_n}(x)\notin J_{m_n}$ and $|J_{m_n}|\ge \delta/4.$ Let $\hat{J}_{m_n}\subset J_{m_n}$ be a maximal arc with $|f^m(\hat{J}_{m_n})|\le\alpha$ for $-m_n\le m.$ Observe that $f^{-m_n}(\hat{J}_{m_n})\subset \ell^c_{r_n}$, is closed and does not contain $x.$ Therefore $|f^m(\hat{J}_{m_n})|\to_m 0.$ And there must be some $k_{m_n}\ge -m_n$ such that $|f^{k_{m_n}}(\hat{J}_{m_n})|=\min\{\alpha,\delta/4\}.$ Since $r_n\to 0$ then $m_n-k_{m_n}\to \infty.$  If $L$ is an accumulation center arc of $f^{k_{m_n}}(\hat{J}_{m_n})$ we arrive to a contradiction since $|f^j(L)|\le\alpha $ for any $j\in\Z.$ This proves the claim.

\vskip 5pt

From the claim, we can choose  sequences $r_n\to 0$, arcs $\ell^c_{n}\subset I^c$ with $|\ell^c_n|=r_n$ containing $x$ and closed arcs $J_n\subset \ell^c_n$ not containing $x$ with $|f^m(J_n)|\not\to_m 0.$ We may assume that $\ell^c_{n+1}\cap J_n=\emptyset.$ Therefore, $J_1,\cdots, J_{m_0}$ are as in the beginning of our proof and we arrive to contradiction. This proves the lemma.
\end{proof}

The next lemma is a classical result regarding expansivity. We state it in terms of center arcs.

\begin{lemma}\label{l.borde}
There exists $N>0$ such that if $I^c$ is a center arc with $|f^j(I^c)|\le\alpha$ for all $0\le j\le n$ and exists $j_0$ with $|f^{j_0}(I^c)|=\alpha$ then $j_0\le N$ or $n-j_0\le N.$
\end{lemma}

\begin{proof}
Otherwise, there exist $m_n\to\infty$, $I^c_n$ and $j_n$ such that
\begin{itemize}
\item $|f^j(I^c_n)|\le\alpha$ for all $0\le j\le n.$
\item $|f^{j_n}(I^c_n)|=\alpha.$
\item $j_n\to \infty$ and $n-j_n\to \infty.$
\end{itemize}
Then, taking an accumulation arc $J^c$ of $f^{j_n}(I^c_n)$ we get that $|f^j(J^c)|\le \alpha$ for all $n\in\Z.$
\end{proof}

\begin{lemma}\label{l.sup}
Let $I^c$ be a center arc with $\limsup_{n\to +\infty}|f^n(I^c)|=\infty.$ Then
$$\lim_{n\to+\infty}|f^n(I^c)|=\infty.$$
\end{lemma}

\begin{proof}
Arguing by contradiction, assume that there exist $K>0$ and $n_j\to\infty$ such that $|f^{n_j}(I^c)|\le K. $
For each $j$ we may choose a partition $f^{n_j}(I^c)=I^c_1(j)\cup...\cup I^c_{k_j}(j)$  such that
\begin{itemize}
\item $|f^k(I^c_\ell(j))|\le\alpha$ for $0\le k\le n_{j+1}-n_j.$
\item For each $1\le \ell\le k_j-1$ there exists $k_\ell$ such that $|f^{k_\ell}(I^c_\ell(j))|=\alpha.$
\end{itemize}
By Lemma \ref{l.borde} we know  for each $1\le\ell\le k_j-1$ that  $(1)\;k_\ell\le N$ or $(2)\;n_{j+1}-n_j-k_\ell\le N$. Observe that there exists $C>0$ such that if $k_\ell\le N$ then $|I^c_\ell(j)|\ge C\alpha$ and if $n_{j+1}-n_j-k_\ell\le N$ then $|f^{n_{j+1}-n_j}(I^c_\ell(j))|\ge C\alpha.$
Consider
$$A=\{\ell: 1\le \ell\le k_j-1 \mbox{ and } k_\ell\le N\}\mbox{ and }B=\{\ell:1\le \ell\le k_j-1\mbox{ and }n_{j+1}-n_j-k_\ell\le N\}.$$
It follows that $\# A\le \frac{K}{C\alpha}$ and $\# B\le \frac{K}{C\alpha}.$ Therefore, for each $j$ we have
$$|f^n(I^c)|\le 2\frac{K}{C}+1\;\;\;\;\; \forall n, n_j\le n\le n_{j+1}.$$
Thus, $\limsup_{n\to+\infty}|f^n(I^c)|<\infty,$ a contradiction.
\end{proof}

We have all the ingredients to prove Proposition \ref{p.options}: Let $I^c$ be a center arc. If $\limsup_{n\to+\infty}|f^n(I^c)|=\infty$ then by Lemma \ref{l.sup} we know that $\lim_{n\to+\infty}|f^n(I^c)|=\infty.$ On the other hand, if $\limsup_{n\to+\infty}|f^n(I^c)|<\infty$ we get from Lemma \ref{l.lim0} that $\lim_{n\to+\infty}|f^n(I^c)|=0.$
The same holds when $n\to -\infty.$ And observe, arguing as before that cannot happen that $\lim_{n\to\pm\infty}|f^n(I^c)|=0$ by expansiveness. Thus, we have proved that the only three options as in the statement of the proposition are possible.

It remains to prove the last sentence of the proposition. Let $I^c_n$ be a sequence of center arcs converging to $J^c$ and with $\lim_m|f^m(I^c_n)|=0.$ We know that either $|f^n(J^c)|\to_m 0$ or $|f^m(J^c)|\to_m\infty.$ In case of the latter, we arrive to a contradiction. Indeed, if $|f^m(J^c)|\to_m\infty$ then there exists $k_n\to\infty$ such that $|f^{k_n}(I^c_n)|\to\infty.$ Choose a partition $I^c_n=I^c_1(n)\cup...\cup I^c_{m_n}(n)$ such that $|f^k(I^c_j(n))|\le\alpha$ for all $k\ge 0$ and $1\le j\le m_n$ and $|f^{\tilde{k}_j(n)}(I^c_j(n))|=\alpha$ for some $\tilde{k}_j(n)$ for all $1\le j\le m_n-1.$ There must exists $j_n$ such that $\tilde{k}_{j_n}(n)\to \infty.$ Finally, if $L^c$ is an accumulation arc of $f^{\tilde{k}_{j_n}(n)}(I^c_{j_n}(n))$ we get that $|f^n(L^c)|\le\alpha$ for all $n\in \Z,$ a contradiction.



\begin{thebibliography}{PS}

\bibitem[AH]{AH} N, Aoki, K. Hiraide; Topological Theory of Dynamical Systems, \textit{North Holland Math. Library}, \textbf{52}(1994).

\bibitem[D]{D} M. Doucette, Smooth Models for Certain Fibered Partially Hyperbolic Systems, https://doi.org/10.48550/arXiv.2208.07286

\bibitem[dMM]{dMM} V. De Martino, S. Martinchich; Codimension one compact center foliations are uniformly compact, \textit{Ergod. Th. \& Dynam. Sys.} (2020), \textbf{40}, 2349-2367 .

\bibitem[FR]{FR} J, Franks, C. Robinson; A quasi-Anosov diffeomorphism that is not Anosov, \textit{Transactions of the AMS}, \textbf{223}(1976), 267-278.

\bibitem[H]{H} K. Hiraide, A simple proof of the Franks–Newhouse theorem on codimension-one
Anosov diffeomorphisms, \textit{Ergodic Theory and Dynamical Systems} \textbf{21}(2001), 801-806

\bibitem[H1]{H1} K. Hiraide. Expansive homeomorphisms of compact surfaces are pseudo-Anosov. \textit{Osaka J. Math. Soc. Japan}\textbf{ 27 }(1990), 117-162.

\bibitem[H2]{H2} K. Hiraide. Expansive homeomorphisms Expansive homeomorphisms with the pseudo-orbit
tracing property of $n$-tori. \textit{J. Math. Soc. Japan
Vol. 41, No. 3, 1989}\textbf{Vol. 41}(1989), 357-389.

\bibitem[HP]{HP} A. Hammelindl, R. Potrie, Partially hyperbolicity and classification: a survey, \textit{Ergodic theory and dynamcial systems} \textbf{38}(2), 2018, 401-443.


\bibitem[L]{L} J. Lewowicz,  Expansive Homeomorphisms of Surfaces, \textit{Bol. Soc. Bras. de Mat.} \textbf{20} (1989), 113-133.

\bibitem[M1]{M1} Ma\~{n}\'{e}, Ricardo Contributions to the stability conjecture.  \textit{Topology}
\textbf{17}  (1978), no. 4, 383--396.

\bibitem[M2]{M2} R. Mañé, Quasi-Anosov diffeomorphisms and hyperbolic manifolds, \textit{Trans. Amer. Math. Soc.} \textbf{229} (1977) 351-370.

\bibitem[N]{N} S. Newhouse,  On Codimension One Anosov Diffeomorphisms, \textit{American Journal of Mathematics } \textbf{92} (3) (1970), 761-770

\bibitem[PS]{PS} A. Passeggi, M. Sambarino, Examples of minimal diffeomorphisms on $\mathbb{T}^2$
semiconjugated to an ergodic translation, \textit{Fund. Math.}, \textbf{222} (2013), 63-97


\bibitem [P]{P} R. Potrie, Robust dynamics, invariant structures and topological classification. \textit{Proceedings of the International Congress of Mathematicians—Rio de Janeiro 2018}. Vol. III. Invited lectures, 2063–2085, World Sci. Publ., Hackensack, NJ, 2018.

\bibitem[R]{R} Ruggiero, R. : On a conjecture about expansive geodesic flows. \textit{Ergodic Theory and Dyn. Sys.} \textbf{16} (1996) 545-553.

\bibitem[V1]{V1} Vieitez, José L. Expansive homeomorphisms and hyperbolic diffeomorphisms on 3-manifolds. \textit{Ergodic Theory Dynam. Systems }\textbf{16 }(1996), no. 3, 591-622.

\bibitem[V2]{V2}  Vieitez, José L. Lyapunov functions and expansive diffeomorphisms on $3D$-manifolds. \textit{Ergodic Theory Dynam. Systems} \textbf{22} (2002), no. 2, 601-632


\end{thebibliography}

\end{document}